\documentclass[a4paper,12pt,leqno]{amsart}

\usepackage{graphics}
\usepackage{thmtools}
\usepackage{wasysym}
\usepackage{caption}
\usepackage{subcaption}

\usepackage[T1]{fontenc}    

\usepackage{amsthm}
\usepackage{amsbsy,amsmath,amssymb,amscd,amsfonts}
\usepackage[pagebackref=true]{hyperref}
\usepackage{url}            
\usepackage{booktabs}       
\usepackage{nicefrac}       
\usepackage{microtype}      

\usepackage{graphicx,float,latexsym,color}
\usepackage[font={scriptsize,it}]{caption}

\usepackage{makecell}

\usepackage[dvipsnames]{xcolor}

\newtheorem{theorem}{Theorem}

\newtheorem{corollary}{Corollary}
\newtheorem{lemma}{Lemma}
\theoremstyle{remark}

\theoremstyle{definition}

\hypersetup{
    pdftoolbar=true,        
    pdfmenubar=true,        
    pdffitwindow=false,     
    pdfstartview={FitH},    
    colorlinks=true,       
    linkcolor=OliveGreen,          
    citecolor=blue,        
    filecolor=black,      
    urlcolor=red           
}

\usepackage{lineno}

\usepackage{comment}
\usepackage{microtype}

\title{Lemoine point and the inscribed conic
}
\author{Liliana Gabriela  Gheorghe}
\begin{document}
\maketitle

\textbf{Abstract.} 
\small{The center of an inscribed conic  which have a given perspector is the complement of its isotomic conjugate. We provide  a synthetic  proof, based on fine proprieties of Lemoine point.
}

   \begin{figure}
\centering
\includegraphics[trim=200 400 200 300,clip,width=0.9\textwidth]{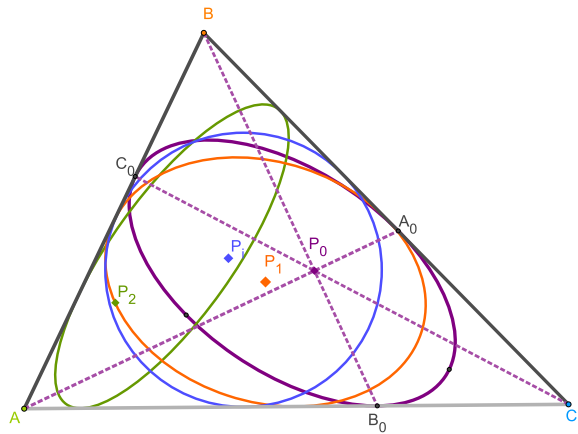}
\caption{In any triangle, there are infinitely many inscribed conics: for each point $P$ there exists a unique conic that
tangents the triangle's sides precisely at the feet of cevians through $P,$ conic's perspector. What about its center? } 
\label{fig:isca}
\end{figure}

\section{Introduction}

When a conic  touches the  sides of a triangle, 
 the lines $AA',BB',CC'$ joining the vertices with the tangency points
 meet at a point $P,$ the perspector. This  is nothing but a manifestation  of Brianchon's theorem, and this concurrency can be  validated by a tandem-use of Carnot and Ceva's theorems.

There exist one (and only one) inscribed conic, whose perspector coincide with its center: the Steiner conic, the i-conic centered at $G.$ 
In any other case,  the center  of the inscribed  conic is distinct from its perspector.  
An elementary geometric construction of the former  from the latter is  quite simple (see Figure \ref{fig:iscacentro} for "a proof without words").
By contrast,  to (geometrically) obtain the perspector of a conic whose center is known, is not an easy task: it requires a careful use  Brianchon's theorem.

 \vspace{0.6cm}

{\bf{Keywords:}} {Lemoine point, inscribed conic, perspector, isotomic conjugate, complement.}

{\bf 2010 Mathematics Subject Classification: 51A05, 51A30, 51M15.}

In this paper, we give a  synthetic proof for the fact that the center of a conic is the complement of its isotomic conjugate (see Figure \ref{fig:20070_iconics}). 
This simple and rigid relation between the two, make them
even: one can be straightforwardly obtained  from the  other.
 
 We  unveil the  relation between the perspector and the center of a conic, by  connecting them  to Lemoine point and the orthic conic.
 The elementary path we adopt allows to give a fully elementary proof of the Grinberg-Yiu theorems, listed as 
Theorem 2.1 and Theorem 2.4 in [MM1].

\subsection*{Main results} We give a geometric proof for the following  facts.

  \begin{figure}
\centering
\includegraphics[trim=50 90 20 30,clip,width=0.9\textwidth]{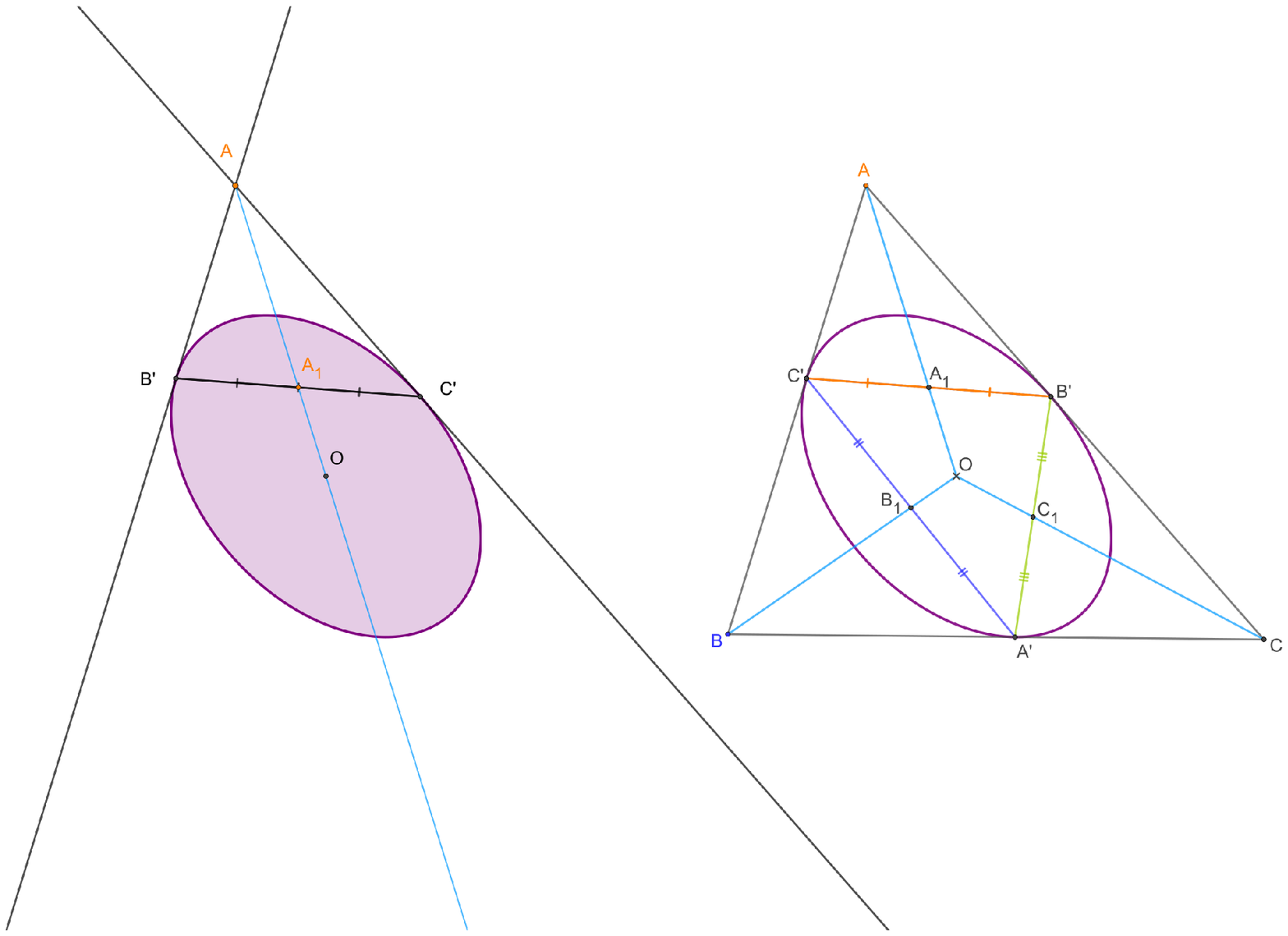}
\caption{{\bf{Left}} The line that join the midpoint of the chord of a conic, with its pole, pass through the center of the conic.  {\bf{Right}} If $A',B',C'$ are the tangency points of an inscribed conic into $\triangle{ABC}$, then the lines that join the vertices  $A$, $B$, and $C$,  to the midpoints of $\triangle{A'B'C'}$' sides, meet at the center of the conic.} 
\label{fig:iscacentro}
\end{figure} 

   \begin{figure}
\centering
\includegraphics[trim=250 500 150 300,clip,width=0.9\textwidth]{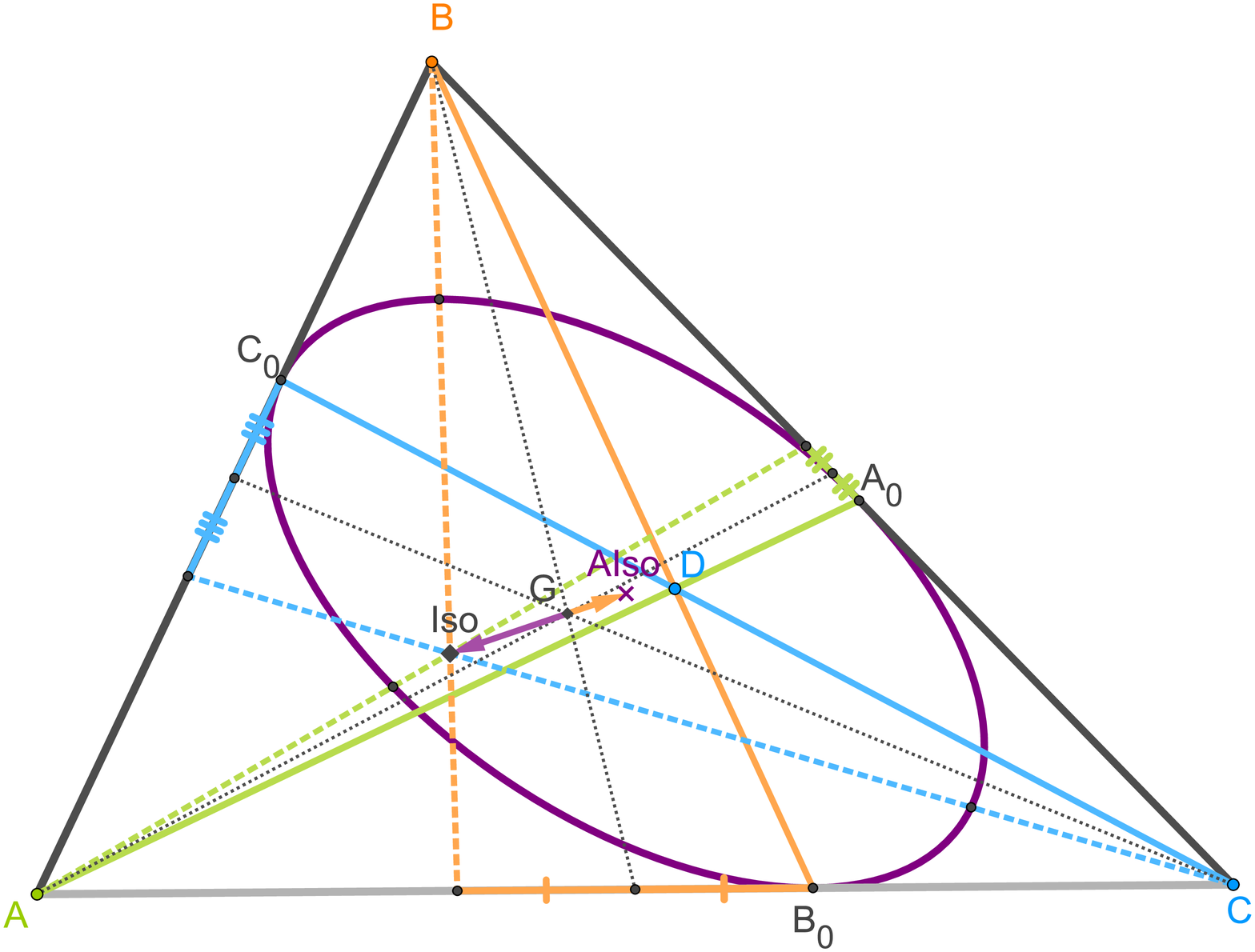}
\caption{I) If $D$ is a point and $A_0$,$B_0$,$C_0$ the feet of cevians through $D$,  the isotomic conjugate of $D$, here noted by $Iso$, is the intersection of cevians through the reflections of $A_0$, $B_O$, $C_O$ into the midpoints of the sides. II) The complement of Iso is the point $A Iso$ such that 
$\vec{G A Iso}=-1/2\vec{G Iso}$.  III) The center of an inscribed conic is the complement of the isotomic conjugate of its perspector.} 
\label{fig:20070_iconics}
\end{figure}

\begin{theorem}
The orthic conic is centered at Lemoine point.
\label{thm:ortic}
\end{theorem}

\begin{theorem} Lemoine point 
is the complement of the isotomic conjugate of the orthocenter.
\label{thm:anticomplement}
\end{theorem}

\begin{theorem}
The center of an inscribed conic   is the complement of the isotomic conjugate of its perspector.
\label{thm:perspector}
\end{theorem}

\subsection*{Related work} Inscribed conics, in their classic setting appear in almost any classic book on conics;  see [AZ],[C],[G], [GSO].
Inscribed conics in  Poncelet pairs,  are in  [RG],[GRK],[G1].
Inscribed conic in trilinear coordinates, profoundly related both to  isotomic or isogonal transform of lines are in the recent paper [A].
 A sophisticated  study of the relation between perspector, orthocenter and the center of a conic, that  uses a mix of
projective and afine techniques is in  [MM1] and  [MM2].

 \section{Under-exploited proprieties of Lemoine point}

  Lemoine geometry is nowadays  a consolidate   chapter in triangle's geometry and 
  the point itself  is 
still of interest  (see  [L],[M] or the more recent  [P] and the references therein).
Curiously enough, the connection between geometric proprieties of Lemoine point and inscribed conics or hexagons is scarcely exploit.
Here, we  use several  proprieties of the Lemoine point, that reconnect it to its natural environment: hexagons with parallel sides and Lemoine circle

 For the convenience of the reader, we begin by recalling some useful observations on Lemoine point, all  of them classic .

\begin{lemma}
 In any triangle, the lines that join the  vertices to the midpoints of its orthic  meets at Lemoine's point.  
 \label{lemma:simedian_mid} 
\end{lemma}
 \begin{proof} Let $A_{m}$, $B_{m}$, and  $C_{m}$ the midpoints of the orthic triangle (see Figure \ref{fig:ortico}).
 The sides of the orthic  are anti-parallel to the sides of $\triangle{ABC}.$ Since the locus of the midpoints of the anti-parallels at a triangle side is the symmedian that corresponds to the referred side, the lines $AA_{m}$, $BB_{m}$, $C_{m}$,  are therefore
 $\triangle{ABC}$'s
 symmedians,  meeting at Lemoine point. 
  \end{proof}
The next result is also classic. 
\begin{lemma}
Let $BC$ be a chord of a conic centered in $O$ and $M$ be its midpoint. If the tangents to the conic at $B$ and $C$ meet at a point $A,$ then 
$AM$ pass through the center of the conic.
\label{lemma:polo_tang}
\end{lemma}
\begin{proof} Refer to Figure \ref{fig:iscacentro}.
Perform an afine transform that maps the conic into a circle. Afine transform preserves the midpoint. Therefore, the original claim reduces to the following (familiar) statement: if the tangents at the endpoints of a chord
of a circle $B'C'$ 
meet at $A',$ then the line $A'M'$ that  join 
$A'$ to the midpoint of $B'C'$ pass through the circle's center.
\end{proof}

Finally,
Lemoine point and the altitudes are interlinked by a  curious   geometric relation; refer to Figure \ref{fig:Isotomico}.
  \begin{figure}
\centering
\includegraphics[trim=100 80 100 20,clip,width=0.9\textwidth]{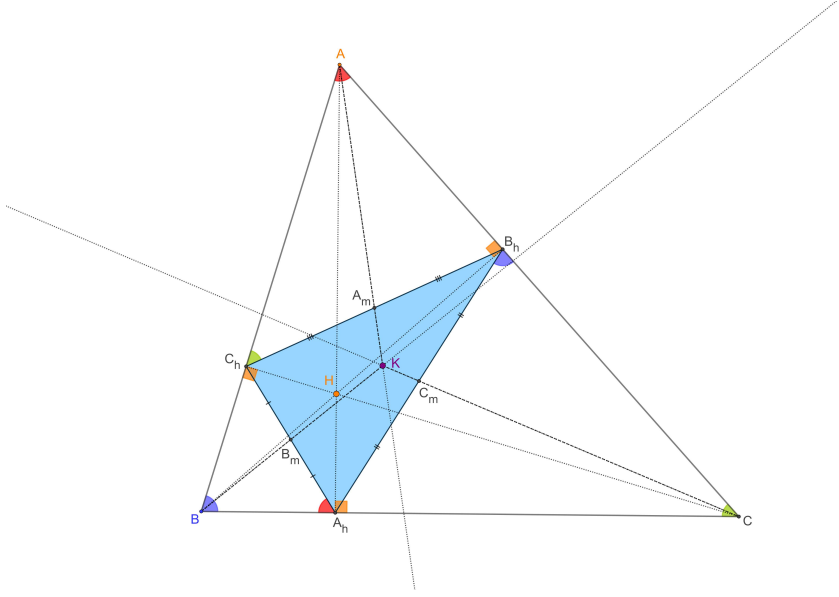}
\caption{The sides of the orthic (light blue) are antiparallel to those of $\triangle{ABC}$. The line that join
$A,B,C,$ to the midpoints $A_m,B_m,C_m$ of  the orthic triangle, meet at Lemoine point.} 
\label{fig:ortico}
\end{figure}

\begin{lemma} In any triangle, the segment that join the midpoint of one side, to the midpoint of the corresponding altitude  passes through Lemoine point.
 \label{lema:Isotomico}
 \end{lemma}
 For a proof refer to  [H].

\section{Proofs of the main results}
\subsection{Proof of Theorem \ref{thm:ortic}}.

\begin{proof}
The orthic conic tangents triangle's sides at the feet of the altitudes. 
By  Lemma \ref{lemma:simedian_mid}, the lines that join the triangle's vertices, to  the midpoints of the orthic sides, meets at Lemoine point. On the other hand, by Lemma  \ref{lemma:polo_tang} the same  point is (also) the center of the orthic conic,  ending the proof. 
\end{proof}

    \begin{figure}
\centering
\includegraphics[trim=100 160 160 50,clip,width=0.9\textwidth]{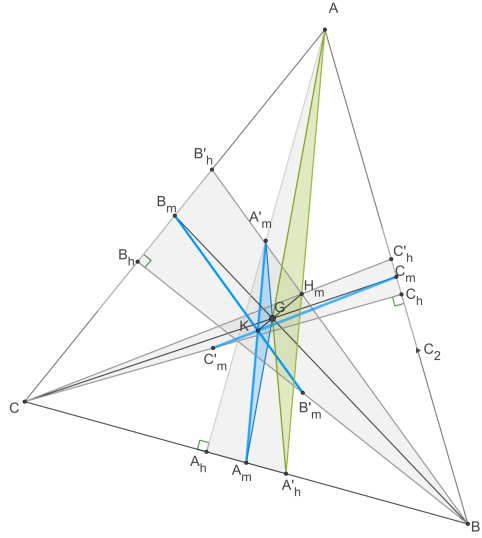}
\caption{I) The segments $A_m A'_m,B_m B'_m,C_m C'_m$ joining the midpoints of an altitude with the midpoint of the corresponding side, concur at Lemoine point.  
II) $A'_h$, $B'_h$, $C'_h$ are the reflections of the feet of the altitudes 
$A_h, B_h, C_h$ into the midpoints; 
$H_m$, the isotomic conjugate of the orthocenter
is the intersection of $A A'h, B B'_h, C C'_h$.
II) $K$ shown as  the complement of $H_m,$ the isotomic conjugate of the orthocenter. } 
\label{fig:Isotomico}
\end{figure}

\subsection{Proof of Theorem  \ref{thm:anticomplement}}.

\begin{proof} Refer to Figure ~\ref{fig:Isotomico}.
Let $H_m$ be the isotomic conjugate of the orthocenter $H.$ 
We shall to prove that
$$\overrightarrow{KG}=\frac{1}{2}\overrightarrow{GH_m},$$ where $G$ is the barycenter. Let $A_m',B_m',C_m'$ be the midpoints of the altitudes;
by Lemma ~\ref{lema:Isotomico}, the lines $A_m A_m',$ $B_m B_m'$ and 
$C_m C_m'$ meet at Lemoine point. Let $A'_h$  the reflection of $A_h,$ in $A_m,$ the midpoint of $BC;$ similarly define $B_h',C_h'.$
Then the lines 
 $A A_h',B B_h',C C_h'$ intercept in $H_m,$ the isotomic conjugate of the orthocenter $H.$
 
 Denote (temporarily) by $G'$ the intersection of $A A_m$ and $A'_m A'_h.$

Note that $$\triangle{A_m G' A_m'}\sim 
\triangle{A G' A'_h}$$ since by construction the angle in the common vertex $G'$ is the same,  and
$A_m A_m'$ is a mid-base in 
$\triangle A A_h A_h';$ hence 
\begin{equation}
\overrightarrow{G' A_m}=-\frac{1}{2}\overrightarrow{G' A},\;\;
\overrightarrow{A_m A'_m}=-\frac{1}{2}\overrightarrow{A' A_h},
\label{caso2}
\end{equation}
thus, $G'=G$.

This proves that triangles $\triangle{KGA_m}\sim\triangle{H_m GA}$, since they are inversely homothetic, of ratio $-\frac{1}{2}$ with respect to  homothety center $G.$ Hence
$$\overrightarrow{GK}=-\frac{1}{2}\overrightarrow{GH_m},$$
showing that the point $K$ itself is the complement of $H_m,$
 the isotomic conjugate of $H.$

\end{proof}

  
 Now let us consider the "Lemoine hexagon" the hexagon
 $[A_1 A_2 B_1 B_2 C_1 C_2]$ whose diagonals are the three congruent antiparallels at triangle's sides (see Figure  \ref{fig:20021_inscript}). 
 It is
 inscribed into Lemoine circle, centered at $K$, and its vertices are the intersection of $\triangle{ABC}$'s sides, with the antiparalels through $K$ at its sides.

\begin{lemma} 
 The hexagon $[A_1 A_2 B_1 B_2 C_1 C_2]$   is inscribed into Lemoine's second circle and circumscribes
 the orthic conic of $\triangle{ABC}.$
 \label{lemma:ortic_hexagon}
\end{lemma}

\begin{proof}
Due to central symmetry, the opposite sides of $[A_1A_2 B_1 B_2 C_1 C_2]$ are parallel and congruent. Since the  side $A_1A_2$ is tangent to the orthic conic, which is concentric with Lemoine circle, so is the side $B_2 C_1$. Hence 
$[A_1 A_2 B_1 B_2 C_1 C_2]$ circumscribes the orthic conic and tangents it at the feet of the altitudes of $\triangle{ABC}.$
\end{proof}

   At this point, we learned  two  facts on orthic conic: 
  \begin{itemize}
      \item it is centered at Lemoine point $K;$ 
      \item $K$ is (also) the complement of the isotomic conjugate of $H.$
      
  \end{itemize}

\begin{figure}\centering
\includegraphics [trim=0 300 0 200,clip,width=0.9\textwidth]{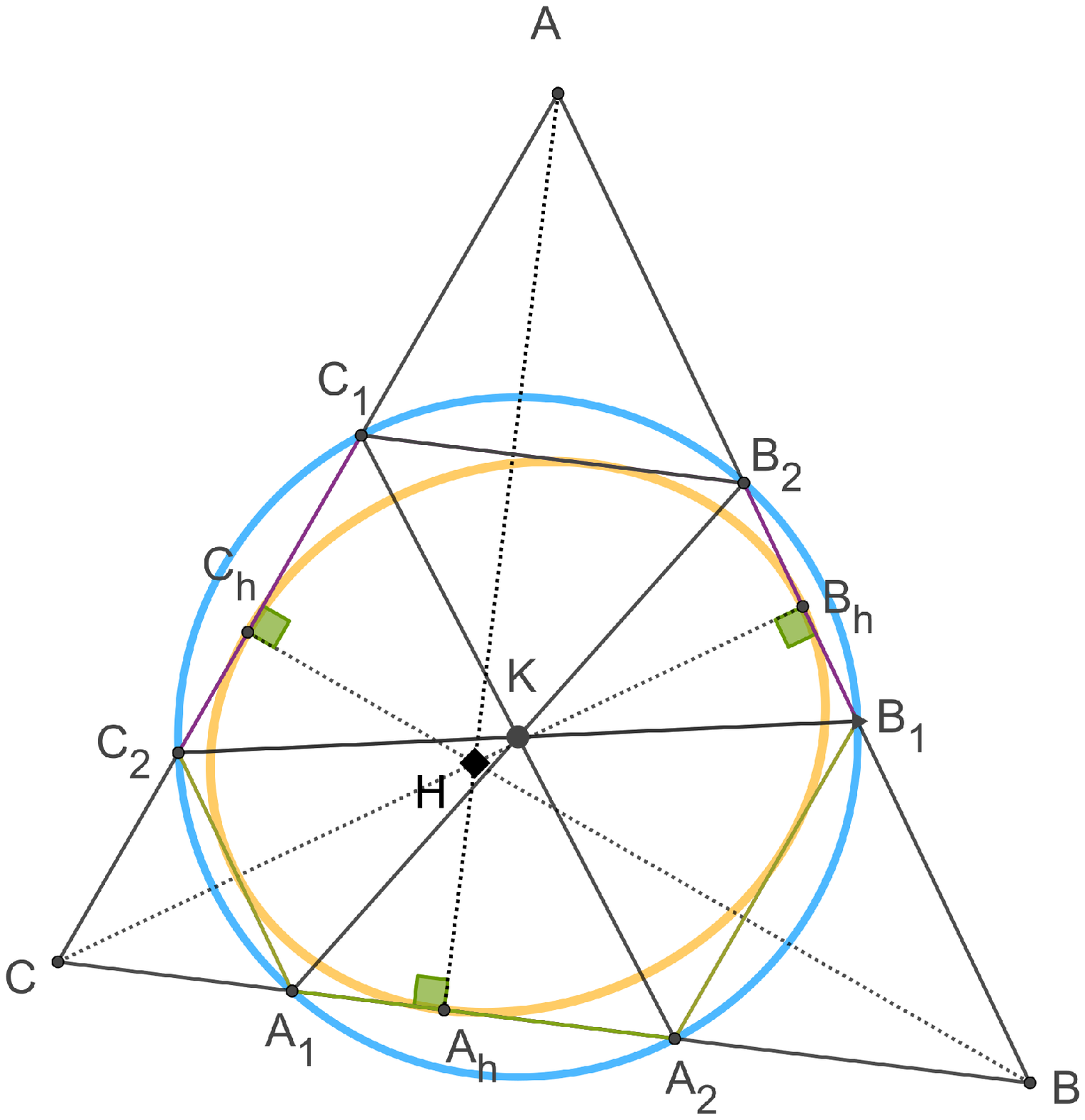}
\caption{I)
Lemoine circle (blue) is centered at the Lemoine point, $K.$ The tree diameters $A_1 B_2, A_2 C_1, B_1 C_2$ are antiparallel with triangle's sides.
II) Hexagon $[A_1 A_2 B_1 B_2 C_1 C_2]$ is inscribed into Lemoine circle and circumscribes the orthic conic of $\triangle{A B C}$. }
  \label{fig:20021_inscript}
\end{figure}

 \begin{figure}
\centering
\includegraphics[trim=0 300 0 300,clip,width=0.9\textwidth]{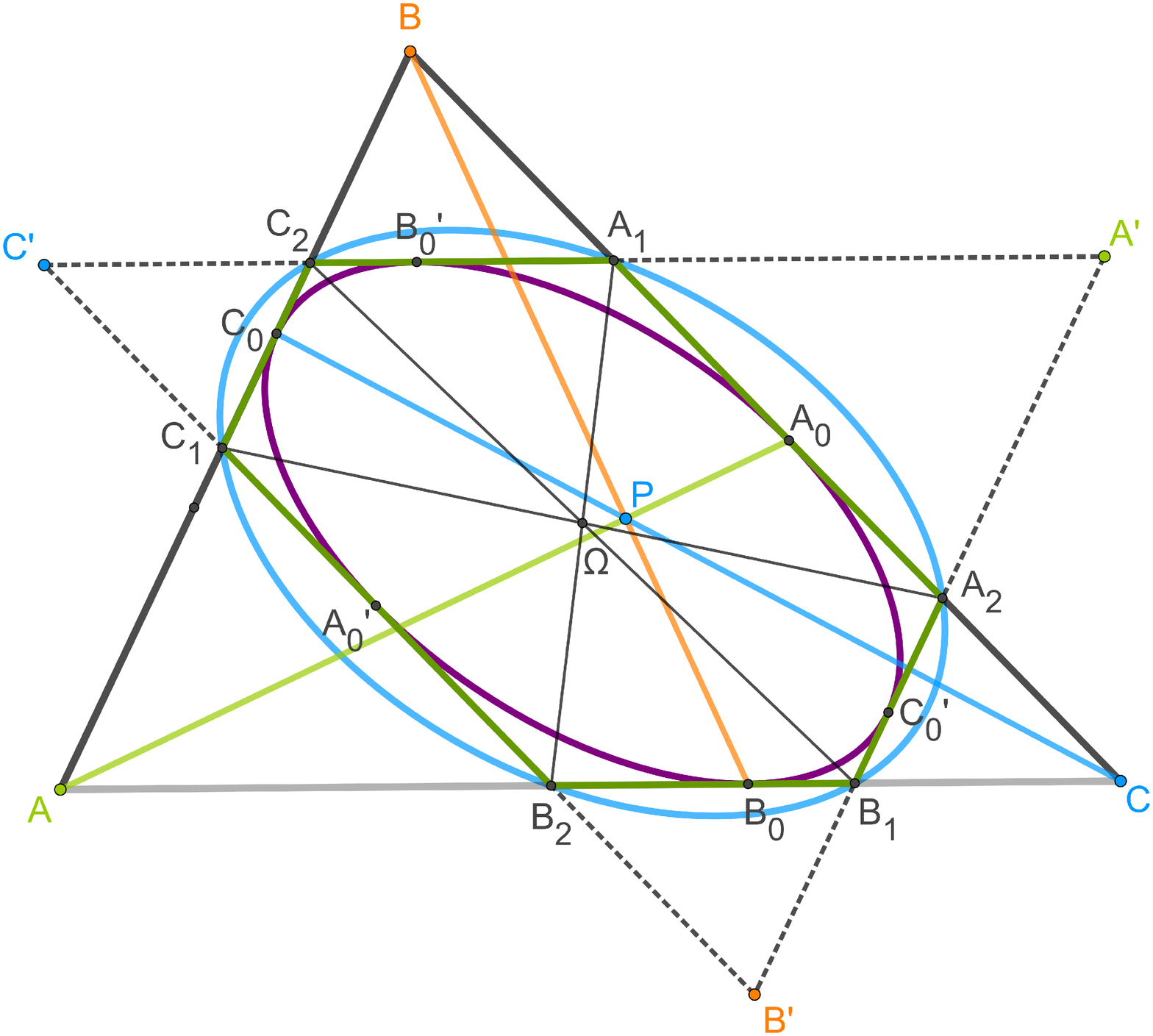}
\caption{$[A_1 A_2 B_1 B_2 C_1 C_2]$, $\triangle{A B C}$,
and $\triangle{A' B' C'}$ share the same inscribed conic, centered at $\Omega$ (purple ellipse). The diagonals meet at the center of the conic  $\Omega$.
By an afine transform, the circumscribed conic of $[A_1 A_2 B_1 B_2 C_1 C_2]$, (blue circumellipse) can be mapped into a circle, hence the point $\Omega, $ is mapped  into Lemoine point.} 
\label{fig:ferramenta_hexagono}
\end{figure}

 \section*{} 
 
  Now, we shall finally prove that (and why) the center of (any)  conic with given perspector is the complement the  isotomic conjugate of the perspector.
  
\begin{lemma} There exists an afine transform that maps an inscribed  ellipse $\gamma$, of perspector $P$,    into the orthic conic.
\end{lemma}  
\begin{proof}   In fact, if $\Omega$ is $\gamma$'s center, then  reflection of $\triangle{ABC}$ in $\Omega$ led to a  $\triangle {A'B'C'}$ which share with the former the same i-conic (see Figure ~\ref{fig:ferramenta_hexagono}). Due to central symmetry, the sides of these two triangles mutually intercept at six points,  vertices of a hexagon with  pairs of opposite parallel and congruent sides. This hexagon also shares with the two reflected triangles  the  i-conic. Since this hexagon have opposite parallel sides, by
 Pascal theorem this hexagon has a  circumellipse, $\Gamma$; the  central symmetry  ensures that $\Gamma$ and   $\gamma$ are concentric.
  
 There exists an afine transform  mapping hexagon's $[A_1 A_2 B_1 B_2 C_1 C_2]$ circumellipse $\Gamma$ into a circle, $\tilde {\Gamma}$ and preserves the  center of the conic. This  maps the perspector $P$ of the conic $\gamma$, into the perspector of the i-conic $\tilde{\gamma}$ of  $\triangle{\tilde A \tilde B \tilde C}$ the image 
 of the referred triangle.
 Thus, 
 by Lemma
~\ref{lemma:ortic_hexagon}, the circle $\tilde{\Gamma}$ is precisely   
 $\triangle{\tilde{A}\tilde{B} \tilde{C}}$'s
 Lemoine circle, centered at its   Lemoine's point $\tilde K.$
 The proof finishes once reminded that the perspector of the i-conic centered at Lemoine point $\tilde{K}$ is the orthocenter.
\end{proof}

 \subsection*{Proof of Theorem 3.}
 \begin{proof}
   Afine transforms preserves colinearity, as well as the proportions between points located on the same line. Therefore, afine transforms preserves the barycenter 
 $G$, isotomic conjugacy and complementarity,  ending the proof.
\end{proof}
As a corollary,  a classic fact on Steiner conic.
\begin{corollary}
There exist one (and only one) inscribed conic whose perspector coincides with its center: this is Steiner conic, centered at $G.$
\end{corollary}
\begin{proof}   The map that associates to a point $P,$ the complement of its isotomic conjugate has a unique fixed point: triangle's baricenter.
 \end{proof}
Finally, note that Lemma \ref{lemma:simedian_mid}, \ref{lemma:polo_tang} and  \ref{lema:Isotomico}  led to a straight-forward synthetic proof for  both 
{Theorem 2.1} and  Theorem 2.4, in [MM1], which are the ingredients of  Grindgberg-Yiu theorem  
(Theorem 2.7 in [MM1]).

\section* {Bibliography}
[A] A. V. Akopyan, {\it Conjugation of lines with respect
to a triangle}, Journal of Classical Geometry, Vol.
1, pp. 23-31, 2012.

[AZ] {A. V. Akopyan and A. A. Zaslavsky},
{\it Geometry of Conics},
{Amer. Math. Soc.},
{Providence, RI},
{2007}

[B] Bradley, C. J. {\it Hexagons with Opposite Sides Parallel}, The Mathematical Gazette, {\bf 90}, no. 517, 2006, pp. 57–67.
 
[C] Casey, J.{\it  A treatise on the analytical geometry of the point, line, circle, and conic sections, containing an account of its most recent extensions, with numerous examples}
 Hodges, Figgis, and Co., Longmans, Green, and Co., 
Dublin London,
1885.

 [G] Gallatly, W. {\it The Modern Geometry of the Triangle} Francis Hodgson,
89 Faebingdon Street, E.G, 1910.

[ETC] Kimberling, C.,
    {\it Encyclopedia of Triangle Centers},
url={bit.ly/3mOOver}

[G1] R. Garcia, {\it Elliptic Billiards and Ellipses Associated to the 3-Periodic Orbits},
{Amer. Math. Monthly},{\bf {126}}, {06},
pp 491--504,
     {2019}.

[GRK] {R.Garcia,  D. Reznik and  J. Koiller},
{\it New Properties of Triangular Orbits in Elliptic Billiards},
{Amer. Math. Monthly},
{\bf{128}},{(10)},
{2021},
pp 898--910.
 
[GSO]
{Georg Glaeser and Hellmuth Stachel and Boris Odehnal},
{\it The Universe of Conics: From the ancient Greeks to 21st century developments},
{Springer Spektrum}, Berlin-Heidelberg,
{2016}.

[H] Honsberger, R.  {\it Episodes in Nineteenth and Twentieth Century Euclidean Geometry}  , The Mathematical Association of America, New Library, Washington, 1995. 

[L] Lemoine, E
{ \it Note sur un point remarquable du plan d'un triangle}
Nouvelles annales de mathématiques, Série 2, Tome 12 (1873), pp. 364-366.
 
[M] Mackay, J. S. (1895), {\it Symmedians of a triangle and their concomitant circles}, Proceedings of the Edinburgh Mathematical Society, 14: 37–103, doi:10.1017/S0013091500031758

[MM1] Minevich, I., Morton, P.,
 {\it Synthetic foundations of Cevian Geometry, I: fixed points of affine maps}
Journal of Geometry Vol 108 (2017), pages 437–455.

[MM2] Minevich, I., Morton, P.,
 {\it Synthetic foundations of Cevian Geometry IV:
the TCC-perspector theorem},
International Journal of geometry
Vol. 6 (2017), No. 2, 61-85.

[P] Pamfilos, P. http://users.math.uol.gr/Pamfilos/
eGallery/problems/
Symmedian Vecten.html

[RG] {D. Reznik and R. Garcia},
{Related By Similarity {II}: Poncelet 3-Periodics in the Homothetic Pair and the Brocard Porism},
{Intl. J. of Geom.},
{\bf 10},
{(4)},
{03},
{2021},
pp {18--31}.

\bigskip
 \bigskip
 \bigskip
 
DEPARTAMENTO DE MATEMÁTICA

UNIVERSIDADE FEDERAL DE PERNAMBUCO

RECIFE, (PE) BRASIL

\textit{E-mail address}:

\texttt{liliana@dmat.ufpe.br}
\bigskip
\bigskip
\end{document}